\documentclass{amsart}

\newtheorem{theorem}{Theorem}[section]
\newtheorem{lemma}[theorem]{Lemma}
\newtheorem{proposition}[theorem]{Proposition}

\newtheorem{corollary}[theorem]{Corollary}
\newtheorem{remark}[theorem]{Remark}
\newtheorem{example}[theorem]{Example}
\newtheorem{definition}[theorem]{Definition}

\numberwithin{equation}{section}

\usepackage{amsmath}
\usepackage{amsfonts}
\usepackage{rotating}
\usepackage{bbm}
\usepackage[all,cmtip]{xy}
\usepackage[utf8]{inputenc}  

\usepackage{amssymb}
\usepackage{graphicx}
\usepackage{dsfont}
\usepackage{color}
\usepackage{hyperref}
\usepackage{amssymb,amsthm,amsfonts,amsmath,amscd,amssymb,epsfig}
\usepackage[all]{xy}

\usepackage{enumerate,array,hyperref, graphicx,color,amsthm,
bbm,mathrsfs,a4wide}
\usepackage{amssymb}   
\usepackage{graphicx}
\usepackage{caption}
\usepackage{setspace}
\usepackage{tikz}

\newcommand{\Q}{{\mathbb{Q}}}

\renewcommand{\epsilon}{\varepsilon}

\begin{document}

\title{The Special Concave Toric Domain for The Rotating Kepler Problem } 
\author{Amin Mohebbi}
\ifpdf


\maketitle

\section*{Abstracts}
The Rotating Kepler Problem (RKP) emerges as a fundamental model in celestial mechanics, particularly as a limiting case of the circular restricted three-body problem. It provides a tractable yet rich framework for exploring periodic orbits, energy levels, and symplectic structures. In this paper, we study the RKP with energy values less than or equal to the critical threshold $-\frac{3}{2}$. We apply two key symplectic regularizations, the Ligon–Schaaf and Levi–Civita transformations, to reveal a bounded component in the phase space of the RKP.

Within this framework, we construct a special concave toric domain (SCTD) associated with the RKP, providing a concrete setting to compute embedded contact homology (ECH) capacities for energy levels below the critical value. The SCTD provides a geometric framework for rigorously analyzing symplectic embedding problems and energy constraints in dynamical systems. Finally, we introduce a novel combinatorial tree structure, inspired by the Stern–Brocot tree, which encodes the energy data and facilitates the computation of ECH capacities on this bounded component. Our results contribute to the broader understanding of symplectic embeddings in celestial mechanics and offer new tools for the study of Hamiltonian dynamics in rotating systems.

\section{Introduction}
In this paper, we study the Rotating Kepler Problem (RKP) and introduce a bounded component that arises when the critical energy value is less than or equal to  $-\dfrac{3}{2}$.
This bounded region becomes evident through a sequence of symplectic regularizations applied to the RKP Hamiltonian. We begin with the Ligon–Schaaf regularization, which can be viewed as a global extension of the Belbruno–Moser–Osipov regularization. Below the first critical energy level, these regularizations transform the bounded components of the planar circular restricted three-body problem into fiberwise star-shaped hypersurfaces within the cotangent bundle over the 2-sphere $T^\ast S^2$.
This approach involves interchanging the roles of base and fiber—that is, interpreting momentum as position and position as momentum—and then compactifying the phase space by adding a fiber at infinity, corresponding to collisions where momentum becomes unbounded.

In this work, we focus on a particular case of the Ligon–Schaaf regularization adapted to two dimensions and negative energy levels. This transformation maps negative-energy trajectories of the Kepler problem to geodesics on the 2-sphere $S^2$, while preserving the underlying symplectic structure.
We then apply the Levi-Civita regularization, which further transforms the system into a pair of uncoupled harmonic oscillators with periodic flow. Together, these regularizations reveal the bounded component as a symplectic image within a compact domain.

In Chapter 2 (Preliminaries), we present the foundational definitions and concepts required throughout the paper. We briefly review the Restricted Three-Body Problem (R3BP), its Hamiltonian, and the classical Kepler Problem. We introduce the basics of symplectic manifolds, symplectomorphisms, and integrable systems, along with examples. The chapter concludes with a discussion of different orbit types in the Kepler problem (e.g., circular, elliptical, periodic, and synodic), and we specify the energy levels associated with each type of orbit.

In Chapter 3, we focus on the Hamiltonian formulation of the Rotating Kepler Problem. This Hamiltonian is expressed as the sum of the Keplerian Hamiltonian and the angular momentum term. We demonstrate that the RKP is an integrable system, identify its critical energy value, and investigate the behavior of orbits based on energy levels using the Runge–Lenz vector.

Chapter 4 introduces the Ligon–Schaaf regularization in detail. Since we focus on negative energy levels in two dimensions, we use a special case of this regularization. We show that it is a symplectomorphism, preserving the symplectic structure of the phase space. This allows us to map solutions of the RKP to geodesics on the sphere $S^2$. We then apply the Levi-Civita regularization, resulting in a mapping of the solutions to the cotangent bundle $T^\ast S^2 \setminus S^2$. The resulting Hamiltonian flow is periodic and corresponds physically to two uncoupled harmonic oscillators.

In Chapter 5, we construct a special concave toric domain (SCTD) for the RKP. We provide a detailed computation showing the existence of a convex function on the cotangent bundle of $S^2$ after applying both the Ligon–Schaaf and Levi-Civita regularizations. We define a closed region in $\mathbb{R}^2$ in which the SCTD is realized. This domain is then compared to the concave toric domain defined by Hutchings. Considering the critical energy $-\dfrac{3}{2}$, we explore three distinct scenarios:

When the energy is less than or equal to the critical value, both bounded and unbounded components exist, and the SCTD corresponds to the bounded one.

When the energy exceeds the critical value, only an unbounded component exists, and the SCTD cannot be defined.

Finally, in Chapter 6, we introduce a new tree structure, inspired by the Stern–Brocot tree. After reviewing the classical Stern–Brocot construction, we propose a modified version tailored for our purposes. This new tree plays a key role in computing Embedded Contact Homology (ECH) capacities for the RKP—an analysis to be developed in future work. The RKP is essentially the Kepler problem in the rotating coordinate system and is significant because it represents a limiting case of the restricted 3-body problem. The results obtained from this computation are important for understanding symplectic embedding in the restricted 3-body problem (R3BP).

\section{Preliminaries}
The Three Body Problem (3BP) is a 12-dimentional phase space that can be reduced by using symmetries and the presence of three-body collisions, which can't always be regularized. The Hamiltonian of the 3BP is given by 
\begin{align}
    H =  \Sigma_{i=1}^3 \dfrac{1}{2m_i} |p_i|^2 - \Sigma_{i<j} \dfrac{m_i m_j}{|q_I - q_j|}
\end{align}
and is a function on $T^* (\mathbb{R}^6 / \Delta)$ where 
\begin{align}
    \Delta \{(q_1, q_2, q_3 ,p_1, p_2, p_3) \in \mathbb{R}^6 | q_i \neq q_j , p_i \neq p_j \textit{for} i \neq j \}
\end{align}
we think of $(q_i, p_i)$ as a point in $T^* \mathbb{R}^6$ where $q_i$ is the coordinate representation on the manifold and $p_i$ is the dual coordinate basis for $T^* \mathbb{R}^6$. 

The restricted three-body problem (R3BP) is a particular case of 3BP. If we denote the weight of the third body by $m_3$ and restrict it to zero, i.e., take the limit $m_3 \longrightarrow 0$, then we obtain a restriction of the 3BP, which is called the R3BP. 

We know that the space of all covectors at point $p$ is a vector space called the cotangent space at $p$, in terms of linear algebra, it is the dual space to $T_pM$, the vector space at point $p$ on the manifold $M$. The union of all cotangent spaces at all points of $M$ is a vector bundle called the cotangent bundle, and we denote it as 
\begin{align}
    T^* M = \bigsqcup_{p \in M} T_p^* M.
\end{align}

1 : \textbf{Symplectic  manifold:} A symplectic manifold is a pair $(M, \omega)$ where $\omega$ is a manifold and $\omega \in \Omega^2(M)$ is a two-form satisfying the following conditions
\begin{itemize}
    \item[I: ] $\omega$ is closed. \\
    \item[II: ] $\omega$ is non-degenerate. 
\end{itemize}
This two-form is called the symplectic structure on $M$. 

The assumption non-degeneraty implies that a symplectic manifold is even dimensional i.e., an odd  dimensional manifold never admits a symplectic structure.

2: \textbf{Symplectomorphism:  }
Assume asymplectic manifolds $(M_1, \omega_1)$ and $(M_2, \omega_2)$- A symplectomorphism $\phi: M_1 \longrightarrow M_2$ is a diffeomorphism satisfying $\phi^+ \omega_2 = \omega_1$. 

For example, consider the cotangent bundle $T^\ast \mathbb{R}^n$ with a global coordinate $(q,p) = (q_1, \cdots , q_n, p_1, \cdots, p_n)$ and a symplectic form $\omega = \sum_{i=1}^n dq_i \wedge dq_i$, then 
\begin{align}
    \sigma:  T^\ast \mathbb{R}^n \longrightarrow T^\ast \mathbb{R}^n, \qquad \qquad
    (q,p)  \longmapsto (-p, q) 
\end{align}
is a linear symplectomorphism on $\mathbb{R}^n$.

\begin{theorem}[Noether's theorem] 
Assume that $(M, \omega)$ is a closed symplectic manifold and $F, G \in (M, \mathbb{R})$. Then the following are equivalent.
\begin{itemize}
    \item[(i): ] $G$ is an integral for the flow of $F$, that is, $[X_F, X_G] = 0$. \\
    \item[(ii): ] The flow of $G$ is a symmetry for $F$, that is, $F(\phi_F^t(x))$ is independent of $t$ for every $x \in M$. \\
    \item[(iii): ] $F$ and $G$ Poisson commute, that is, $\{F , G\} = 0$. \\
    \item[(iv): ] The flow of $X_F$ and $X_G$ commute, that is, $[X_F , X_G] = 0.$
\end{itemize}   
\end{theorem}
\begin{proof}
    \cite{key-3}
\end{proof}
The following theorem is the Hamiltonian version of Noether's theorem.
\begin{theorem}
    Suppose that $\mathfrak{g} = Lie(G)$ is a Lie group acting Hamiltonianly on a symplectic manifold $(M, \omega)$. If $H:\longrightarrow \mathbb{R}$ is a Hamiltonian that is invariant under $G$, then each $\xi \in \mathfrak{g} = Lie(G)$ gives an integral $H_\xi$ of $X_H$, or equivalently $\{H , H_\xi \} = 0$.
\end{theorem}
\begin{proof}
    \cite{key-3}
\end{proof}

\textbf{Hamiltonian System and Integrable System:} Let  $ L = L(t,x,v)$ be twice continuously differentiable function of $2n+1$ variables $(t, x_1, \cdots, x_n, v_1, \cdots , v_n)$ where $v \in \mathbb{R}^n = T^* \mathbb{R}^n$ as a tangent vector as the point $x \in \mathbb{R}^n$ that represents velocity $\dot{x}$. 

Consider the minimization problem of the integral action. 
\begin{align}
    I(x) = \int_{t_0}^{t_1} L(t, x, \dot{x}) dt
\end{align}
over the set of paths $ x \in ([(]t_0, t_1], \mathbb{R}^n)$ which satisfy the boundary condition 
\begin{align}
    x(t_0) = x_0 , \qquad \qquad x(t_1) = x_1.
\end{align}
The function $L$ is called the Lagrangian of the variational problem and a continuously differentiable path $x:[t_0, t_1] \longrightarrow \mathbb{R}^n$is called minimal if 
\begin{align}
    I(x) \leq I(x+\epsilon)
\end{align}
for every variational $\xi in C^1([t_0, t_1],  \mathbb{R}^n)$ with $\xi(t_0) = \xi (t_1) = 0$.

\begin{lemma}
    A minimal path $ x : [t_0, t_1] \longrightarrow \mathbb{R}^n$
    is a solution of the Euler-Lagrange equation 
    \begin{align} \label{p1}
        \dfrac{d}{dt} \dfrac{\partial L}{\partial v} = \dfrac{\partial L}{\partial x},
    \end{align}
    where 
\begin{align*}
\dfrac{\partial L}{\partial v} & = (\dfrac{\partial L}{\partial v_1}, \dfrac{\partial L}{\partial v_2}, \cdots , \dfrac{\partial L}{\partial v_n} ) \in \mathbb{R}^n \\
\dfrac{\partial L}{\partial x} & = (\dfrac{\partial L}{\partial x_1}, \dfrac{\partial L}{\partial x_2}, \cdots , \dfrac{\partial L}{\partial x_n} ) \in \mathbb{R}^n.
\end{align*}

 \end{lemma}
\begin{proof}
    Proof in \cite{key-6}.
\end{proof}

Consider the Legendre condition 
\begin{align} \label{p2}
    det(\dfrac{\partial^2 L}{\partial v_j \partial v_k}) \neq 0, 
\end{align}
Under this condition, the equation \ref{p1} defines a regular system of second-order ordinary differential equations. 
We can use the Legendre transformation and produce a first-order differential equation system in $2n$ variables and define new variables as follows
\begin{align} \label{p4}
    y_k = &  \dfrac{\partial L}{\partial v_x}(x,v), \qquad\qquad k = 1,2, \cdots , n \\
    \dot{y_k} = & \dfrac{d}{dt}\dfrac{\partial L}{\partial v_k} = \dfrac{\partial L}{\partial x_k}
\end{align}
where $x$ is a solution of \ref{p1}.

From the implicit function theorem and \ref{p2}, we an express $v$ locally as a function of $t, x$ and $y$ and denote it by 
\begin{align}
    v_k = G_k(t, x, y).
\end{align}

\begin{definition} \label{p3}
    The function $H$ is called Hamiltonian as follow 
    \begin{align}
        H(t, x, y) = \Sigma_{j=1} y_j \dot{x}_j - L(x, \dot{x}, t),
        \end{align}
    then 
    \begin{align}
        \dfrac{\partial H}{\partial x_k} = - \dfrac{\partial L}{\partial x_k} , \qquad \qquad \dfrac{\partial H}{\partial y_k} = G_k
    \end{align}
    where t$t, x, y $ are the variables of H. 
\end{definition}

The equation \ref{p1} transforms into the Hamiltonian differential equations
\begin{align} \label{p5}
    \dot{x} = \dfrac{\partial H}{\partial y}, \qquad \qquad \dot{y} = -\dfrac{\partial H}{\partial x}.
\end{align}

\begin{lemma}
    Let  $x: [t_0, t_1] \longrightarrow \mathbb{R}^n$ be a continuously differentiable path and $y: [t_0, t_1] \longrightarrow \mathbb{R}^n$ be a new variable by \ref{p4}. Then $x$ is a solution of the Euler-Lagrange equation  \ref{p1} and if and only if the functions $x$ and $y$ satisfy the Hamiltonian system \ref{p5}. 
\end{lemma}

\textbf{Hamiltonian Flow:} 
Let $z = (x_1,  \cdots, x_n , y_1, \cdots, y_n) \in \mathbb{R}^n$, the Hamiltonian system \ref{p5} can be written as
\begin{align}
    J_0\dot{z} = \triangledown H_t(z),
\end{align}
where $H_t(z) = H(t,z)$, and $\triangledown H_t(z)$ denotes the gradient of $H_t$ and $J_0$ denotes $2n \times 2n$ matrix 

\begin{align}
\begin{bmatrix}
\mathbb{1} & \mathbb{0} \\
\mathbb{0} & \mathbb{1}
\end{bmatrix}.
\end{align}
$J_0$ shows a rotation through $\frac{\pi}{2}$ and $J_0^2 = \mathbb{1}$.

The vector field 
\begin{align}
    X_{H_t} = -J_0 \triangledown H_t : \mathbb{R}^{2n} \longrightarrow \mathbb{R}^{2n}
\end{align}
is called Hamiltonian vector field associated to the Hamiltonian function $H_t$ or the symplectic gradient of $H_t$ i.e., $X_{H_t}$.

Assume a smooth time-dependent Hamiltonian function 
\begin{align}
    \phi_H^{t,t_0} : \mathbb{R} \times \mathbb{R}^n & \longrightarrow \mathbb{R} \\
    (t, x, y) & \longmapsto H_t(x, y)    
\end{align}
is a solution operator of \ref{p5}. It is defined by $\phi_H^{t,t_0}(z_0) = z(t)$ where $z(t)$ is the unique solution of \ref{p5} with initial condition $z(t_0) =z_0$. The domain of $ \phi_H^{t,t_0}$ is an open subset of all $z_0 \in \mathbb{R}^{2n}$ such that the solution exits on the time interval $[t_0, t_1]$ respectively $[t_0, t]$ when $t(t_0)$ and denotes it with $\Omega_{t_0, t}$.

Given the diffeomorphism 
\begin{align}
    \phi_H^{t, t_0} : \Omega_{t,t_0} \longrightarrow \Omega_{t, t_0}
\end{align}
satisfy 
\begin{align}
    \dfrac{\partial}{\partial t} \Omega_H^{t,t_0} = X_{H_t} \circ \Omega_H^{t,t_0}
\end{align}
and 
\begin{align} 
\Omega_H^{t_2,t_1} \circ \Omega_H^{t_1,t_0} =  \Omega_H^{t_2,t_0},  \qquad \qquad \Omega_H^{t_0, t_0} = id.
\end{align}
Note that, if we have a time-independent case i.e., $H_t \equiv H$ and assumption $t_0=0$. The family of $\Phi_H^t = \Phi_H^{t,t_0}$ of diffeomorphisms are called Hamiltoinan flow generated by $H$. 

\textbf{Periodic orbits:}  
The RKP has two kind of orbits, the circular orbits and periodic orbits in the RKP, which re rotating ellipses of positive eccentricity respectively rotating collision orbits. These periodic orbits refereed to as periodic of the second kind. In order that a Kepler ellipse in the inertial system becomes a periodic orbit in the rotating or synodical system its  period has to be a rotational multiple of $2\pi$, i.e.,
\begin{align}
     \tau = \dfrac{2 \pi l}{k},
\end{align}
where $k , l \in \mathbb{N}$ are relatively prime. 

Because of  the rotating coordinate system has period $2 \pi$ with respects to inertial system. While the coordinate system make $l$ turns the ellipse make $k$ turns. 

These periodic orbit never isolated therefore the periodic orbits of the second kind paper in circle families. Thus we can think of them as unparametrized simple orbits. They appear actually in two dimensional torus family. 

By the Kepler's third law, the energy $E_{k,l}$ of the ellipse is determined by its period through the formula 
\begin{align}
    E_{k,l} = - \dfrac{1}{2} (\dfrac{k}{l})^{\frac{2}{3}}.
\end{align}

\section{The Rotating Kepler Problem}
The Kepler problem in a rotating coordinate system is the RKP. It is a special case of the restricted three-body problem (R3BP) where one primary has zero mass. The Hamiltonian of the  Kepler problem is 
\begin{align}
&H:T^*(\mathbb{R}^2\setminus \{0\}) \longrightarrow \mathbb{R} \\ \nonumber
&H(q,p)= \dfrac{1}{2}|p|^2 - \dfrac{1}{|q|},
\end{align}
and the Hamiltonian of the angular momentum is 
\begin{align}
&L:T^*\mathbb{R}^2 \longrightarrow \mathbb{R} \\ \nonumber
& (q,p) \mapsto q_1p_2-q_2p_1,
\end{align}
generate the rotation flow around the origin. Therefore, the Hamiltonian of the rotating Kepler problem is
\begin{align}\label{0.1}
& K:T^*(\mathbb{R}^2 \setminus \{0\}) \longrightarrow \mathbb{R} \\ \nonumber
& K(q,p)= \dfrac{1}{2}|p|^2 - \dfrac{1}{|q|} +q_1p_2-q_2p_1, \qquad \qquad(q,p) \in T^*(\mathbb{R}^2\setminus \{0\}), 
\end{align}
the Hamiltonian system $K=H+L$ is an integrable system in the sense of Arnold-Liouville, \cite{key-3}.
\begin{lemma}
The angular momentum is preserved under the flow of $X_H$. Thus $H$ and $L$ Poisson commute.
\end{lemma} 
\begin{proof}
The standard $SO(2)$ action acts Hamiltonianly on $T^\ast \mathbb{R}^2$ with the momentum map $L$. Thus the Hamiltonian for the central force is $SO(2)$-invariant, so the Noether theorem implies the results.
\end{proof}
Since $H$ and $L$ Poisson commute, we have 
\begin{align}
\{K,L\} = \{H,L\} + \{L,L\}=0. 
\end{align}
Consider the Hamiltonian of the RKP as 
\begin{align}
K(q,p) = \dfrac{1}{2} ((p_1-q_2)^2 + (p_2+q_1)^2)+ \dfrac{-1}{|q|}-\dfrac{1}{2}|q|^2
\end{align}
where we define the effective potential, 
\begin{align}
U:\mathbb{R}^2\setminus \{0\} \longrightarrow \mathbb{R}, \qquad\qquad U(q) =-\dfrac{1}{|q|} - \dfrac{1}{2}|q|^2.  
\end{align}
Thus, 
\begin{align}
K(q,p)=\dfrac{1}{2}((p_1-q_2)^2+(p_2+q_1)^2) + U(q).
\end{align}
\begin{lemma}
The effective potential $U$ of the RKP has a unique critical value $-\dfrac{3}{2}$ and its critical set constant of a circle of radius 1 around the origin. 
\end{lemma}
\begin{proof}
\cite{key-2}.
\end{proof}
Via the projection map $\pi|_{crit(K)}^{-1}(q_1, q_2) = (q_1, q_2, q_2,-q_1)$, where $(q,p) \longrightarrow q$, the critical points of $K$ and $U$ are bijection. Hence, the critical value of $K$ coincides with the critical value of $U$ at the same critical points. 

The  RKP has a unique critical value at $-\dfrac{3}{2}$. Below this critical value, the energy hypersurface has two connected components. The projection of one of the components is bounded in configuration space where the projection of the other one is unbounded.

Consider the Runge-Lenz vector
\begin{align}
A^2= 1+2cL^2,
\end{align}
such that whose length corresponds to the eccentricity of the corresponding Kepler ellipse. 

If we substitute the Hamiltonian \ref{0.1} in the above equation, we have the following inequality, 
\begin{align}
0\leq 1+2H(K-H)^2 = 1+2K^2 H -4KH^2 +3H^3 =: p(K,H).
\end{align}
The equality $p(K, H)=0$ holds if and only if the eccentricity of the corresponding periodic orbit vanishes, i.e. when periodic orbits are circular. 
\begin{remark}
There are two types of periodic orbits in the RKP system: circular and rotating ellipses. To obtain these periodic orbits, we can either fix an energy level $c$ and calculate a family of periodic orbits by changing the mass ratio, or we can fix the mass ratio and vary the periodic orbits while keeping the energy level constant.
\end{remark}

Denote the Kepler ellipse $\epsilon_\tau :[0,\pi] \longrightarrow \mathbb{R}$ where $\pi$ is period. With this relation, we can obtain a solution for the RKP as 
\begin{align}
\epsilon_\tau^{R} = e^{it}\epsilon_\tau (t) 
\end{align}
which is not a longer period. 

On the other hand, the angular momentum $L$ generates the rotation in the q-plane and the p-plane. Thus we have two cases of the orbits. 
\begin{itemize}
\item[i) ] $\epsilon_\tau $ is a circular. In this case, $\epsilon_\tau^R$ is periodic unless it is a critical point when $\tau = 2 \pi $. \\
\item[ii)] $\epsilon_\tau$ is not circle. In this case, it is a proper ellipse or a collision orbit that looks like a line. 
\end{itemize}
If we consider $\epsilon_\tau$ as an ellipse orbit. The resonance relation satisfies in 
\begin{align}
2\pi l = \tau k,
\end{align}
for some positive integer $k $ and $l$. 
\begin{lemma}
Periodic orbits in the RKP of the second kind satisfy the following rotational symmetry 
\begin{align}
\epsilon_{\tau}^R (t+\tau) = e^{2 \pi i l / k} \epsilon_{\tau}^R (t). 
\end{align}
\end{lemma}
\begin{proof}
The resonance condition gives us the equality $\tau = 2 \pi l / k$ and therefore we have 
\begin{align}
\epsilon_{\tau}^R (t+\tau) = e^{it+i\tau } \epsilon_\tau (t+\tau) = e^{2 \pi i l / k} e^{it} \epsilon_\tau (t) = e^{2 \pi i l / k} \epsilon_{\tau}^R (t).
\end{align}
\end{proof}
If we fix $K$, the function will be as follows. 
\begin{align}
p_K:=p(K,\cdot)
\end{align}
is a cubic polynomial in H. Now, considering $H$ to be fixed, we define the following function.
\begin{align} \label{p10}
p^H := p(\cdot ,K)
\end{align}
which is a quadratic polynomial in $K$. Note that, $K= -\dfrac{3}{2}$ is a unique critical value. 

Let $K > -\dfrac{3}{2}$ and denote the root of the cubic polynomial by $R^1(K), \: R^2(K), \: R^3(K)$ in $\mathbb{R}$ with order $R^1(K) < R^2(K) <R^3(K)$. We have the following equalities when $K=-\dfrac{3}{2}$ \cite{key-2} 
\begin{align}
R^1(-\dfrac{3}{2}) =-2, R^2(-\dfrac{3}{2})=R^3(-\dfrac{3}{2}) =-\dfrac{1}{2}.
\end{align}
If $K >-\dfrac{3}{2}$, we can extend $R^1$ to a continuous function on the whole real line such that $R^1(K)$ be unique real root of $p_K$. 

The circular orbits exist only if it holds 
\begin{align}
1+2HL^2 =0. 
\end{align}
The second kind of periodic orbits of the RKP are positive eccentricity respectively rotating collision orbits. 

A Kepler ellipse in the inertial system becomes an orbit in the rotating or synodical system. 
Since the period of the rotting coordinate system is $2\pi$, if the orbit in the rotating system is periodic, the period of the ellipse should be 
\begin{align}
\tau = \dfrac{2\pi l}{k},
\end{align} 
where $k$ and $l$ relatively prime. 
\begin{lemma}
The minimum period $\tau$ of a Kepler ellipse only depends on the energy of a periodic orbit with the equality 
\begin{align}
c_{k,l}= -\dfrac{1}{2}(\dfrac{k}{l})^{\frac{2}{3}}.
\end{align}
\end{lemma}
For the fix Jacobi energy $K$, the angular momentum is $L=K-H$ and we can write 
\begin{align}
A^2=1+2HL^2.
\end{align}
Therefore we can determine the periodic orbit of the second kind corresponding to relatively prime positive integers $k$, $l$, if we know the energy $c$. 

If we consider the Sun-Jupiter system, we can give an astronomical description of a periodic orbit of the second kind as follows.

We denote a torus corresponding to the integers $k$ and $l$ by $T_{k,l}$. Thus using the function $p^H$, \ref{p10}, for a periodic orbit of type $(k,l)$ we have the following relations 
\begin{align}
L_{k,l} =&\sqrt{-\dfrac{1}{2K_{k,l}}} = (\dfrac{l}{k})^{\frac{1}{3}} \\
c_{k,l}^- =& c_{k,l} - L_{k,l} = -\dfrac{1}{2}(\dfrac{k}{l})^{\frac{2}{3}} - (\dfrac{l}{k})^{\frac{1}{3}} = -(\dfrac{l}{k})^{\frac{1}{3}} (\dfrac{k+2l}{2l}) \\
c_{k,l}^+ = & c_{k,l} + L_{k,l} = -\dfrac{1}{2}(\dfrac{k}{l})^{\frac{2}{3}} + (\dfrac{l}{k})^{\frac{1}{3}} = (\dfrac{l}{k})^{\frac{1}{3}} (\dfrac{-k+2l}{2l}) 
\end{align} 
Thus the energy of a periodic orbit of type $(k,l)$ can be considered  as
\begin{align}
c \in (c_{k,l}^- , c_{k,l}^+).
\end{align}
Using the above notation, we can express the interior or exterior periodic orbit of type $(k,l)$ as follows
\begin{itemize}
\item[(i)] If $k=l=1$, the critical value of the RKP is $c_{k,l}^- = \dfrac{3}{2}$ then the exterior and interior direct orbits both collapse to the critical point. \\
\item[(ii)] If $k>l$, then $|L_{k,l}|<1$ then the direct orbit is interior. \\
\item[(iii)] If $k<l$, then $|L_{k,l}|>1$ then the direct orbit is exterior. 
\end{itemize}

Let's tackle the first challenge, which is related to the energy value that falls below $-\dfrac{3}{2}$. In such cases, the RKP's bounded component transforms into a unique concave toric domain when subjected to the Ligon-Schaaf regularization and the Levi-Civita regularization. 

To achieve the objective of the first challenge, we must have a thorough understanding of the Ligon-Schaaf and Levi-Civita regularizations.

\section{Regularization}
\subsection{The Ligon-Schaaf Regularization}
The Ligon-Schaaf regularization is an effective method for regulating collisions. It is a symplectomorphism that maps the solutions of the planar Kepler problem to the geodesics on the sphere $S^2$. Unlike the Belbruno-Moser-Osipov regularization \cite{key-9}\cite{key-10}, the Ligon-Schaaf regularization does not change time. One can think of the Ligon-Schaaf regularization as a global variant of Delaunay variables.

Here we consider the negative energy of the system in dimension $n=2$. However, the Ligon-Schaaf symplectomorphism is applicable for positive energy and any dimension $n$, just like the Belbruno-Moser-Osipov regularization.

Given the form $y \mapsto <x,y>$ on $\mathbb{R}^2$ where $<x,y>$ is the standard inner product. Using this from, we  identify the phase space $P$, i.e. the cotangent bundle of $\mathbb{R}^2 \setminus \{0\}$ with the set of $(q,p)$ such that $q \in \mathbb{R}^2$, $q \neq 0$ and $ p \in \mathbb{R}^2$. 

Now we denote an open subset of $P$ which lives on the negative part of the energy with 
\begin{align}
P_- =\{ (q,p) \in P \: \: | \:\: H(q,p) < 0 \}.
\end{align} 
Take the angular momentum $L=q_1p_2 - q_2p_1$ and write the vector $A$ as $A=(A_1,A_2)$. Then we can have the following equalities 
\begin{align} \label{1.1}
\{L,A_1\} & =-A_2 \\ \nonumber
\{L,A_2\} & =A_1 \\ \nonumber
\{A_1,A_2\} & =-2HL. \nonumber
\end{align}
If we define the eccentricity vector by 
\begin{align}
\eta := \nu A,
\end{align}
where $\nu := (-2H)^{\dfrac{1}{2}}$. Hence we can write the Poisson bracket relation \ref{1.1} in terms of $\eta $ as follow 
\begin{align}
\{L,\eta_1\} =&- \eta_2 \\ \nonumber
\{L,\eta_2\} =& \eta_1 \\ \nonumber
\{\eta_1 ,\eta_2\} =& L.  \nonumber
\end{align} 
If we think of $L$ as $\eta_3$. We can recover precisely the Lie algebra of $SO(3)$. 

We define $J=(L,\eta_1, \eta_2)$ from $P_-$ to the dual of the Lie algebra $SO(3)$ as the momentum map of an infinitesimal Hamiltonian action of $SO(3)$ on $P_-$. Note that if we assume the subalgebra $SO(2)$, then we can extend this infinitesimal action to the standard infinite rotation. 

The Ligon-Schaaf regularization describes how we can map the solutions of the Kepler problem to the geodesics on the sphere $S^2$ in $\mathbb{R}^3$ such that the rotation group $SO(3)$ acts naturally. 

First, we define the phase space for the geodesics on the sphere $S^2$. 
\begin{definition}
The cotangent bundle of $S^2$ can be identify with vectors $(x,y) \in \mathbb{R}^3 \times \mathbb{R}^3$ such that $<x,x> =1$ and $<x,y>=0$. The zero section corresponds to the element $(x,0)$ where $<x,x>=1$. We denote by $T$ the complement of the zero section. 
\end{definition}
We define the angular momentum map of the infinitesimal Hamiltonian action $SO(3)$ on $T$ by 
\begin{align}
\tilde{J}:(x,y) \longrightarrow x \wedge y.
\end{align}       

From the above notation, we show that the images of the Kepler solutions are geodesics with time rescaled under the Ligon-Schaaf map factor such which depends only on the energy. In simpler terms, the Kepler solutions are mapped to the solution curves of the Delaunay Hamiltonian, which is defined as follows:
\begin{align}
\tilde{H}(x,y) = -\dfrac{1}{2} \cdot \dfrac{1}{|y|^2} =-\dfrac{1}{2} \cdot \dfrac{1}{|\tilde{J}|^2}
\end{align}
where $(x,y)\in T$. 

Now we use the above notation and assume that the Ligon-Schaaf regularization is a symplectomorphism that maps the phase space $P_-$ into the phase space $T$ and denotes it by $\Phi =\Phi_{LS}$, then define it as 
\begin{align}
\Phi =& \Phi_{LS} : P_- \longrightarrow T \\ \nonumber 
\Phi(q,p) :=& ((\sin \phi \mathcal{A})+\cos \phi \mathcal{B} , \nu (\cos \phi) \mathcal{A}+\nu (\sin \phi) \mathcal{B}), 
\end{align}
where 
\begin{align}
\mathcal{A}&= \mathcal{A}(q,p) := (|q|^{-1}q - <q,p>p,\nu^{-1}<q,p>), \\ \nonumber
\mathcal{B}&= \mathcal{B}(q,p) := (\nu^{-1}|q|p,|p|^2|q|-1),
\end{align}
and 
\begin{align}
\phi =\Phi_{LS} (q,p ):=\nu^{-1} <q,p>.
\end{align}
The Ligon-Schaaf symplectomorphism has useful properties for computing solutions to the Kepler problem on the sphere $S^2$,.
\begin{itemize}
\item[(i:)]
Let $e_{3}$ be the third standard basis vector in $\mathbb{R}^3$, which is the north pole of the sphere $S^2$. Then $\Phi$ is an analytic diffeomorphism from $P_-$ onto the open subset $T_-$ of $T$ consisting of all $(x,y) \in T$ such that $x \neq e_{3}$. \\
\item[(ii:)] $\Phi $ is a symplectomorphism. \\
\item[(iii:)] If $\gamma$ is a solution curve of the Kepler vector field $X_H$ in $P_-$, then $\Phi \circ \gamma$ is a solution curve of the Delaunay vector field $X_{\tilde{H}}$ in $T$. \\
\item[(iv:)] It holds that $J = \tilde{J} \circ \Phi$. 
\end{itemize}
The Ligon-Schaaf symplectomorphism helps us to define the action of $g$ on $P_-$ as an action on $T$. Let the action $g \in SO(3)$ and denote the obvious action $g$ on $T$ by $g_T$ and the action $g$ on $P_-$ by $g_{P_-}$. Hence we define 
\begin{align}
g_{P_-}(q,p) := \Phi^{-1} \circ g_T \circ \Phi (q,p), \:\:\:\: (q,p) \in P_-.
\end{align}
This is a well-define action. For the map $\Phi$ the identity $J=\tilde{J} \circ \Phi $ holds. 
\begin{proposition}
Suppose $\Phi $ is a map from $P_-$ to $T$. $\Phi$ satisfies $J=\tilde{J}\circ \Phi$ if and only if there exists an $\mathbb{R} / 2 \pi \mathbb{Z}$-valued function $\phi$ on $P_-$ such that $\Phi = \Phi_\phi$. 
\end{proposition}
\begin{proof}
In paper \cite{key-5}.
\end{proof}

\subsection{The Levi-Civita Regularization}
The Levi-Civita regularization is a double cover so that your energy hypersurface becomes $S^3$ instead of $\mathbb{R}P^3$.
 In the language of physics, we can explain this double cover of the geodesic flow on $S^2$ as a Hamiltonian flow of two uncoupled harmonic oscillators. 

We denote the Levi-Civita regularization by $\mathcal{L}$ which is a 2:1 maps from $\mathbb{C}^2 \setminus \{0\}$ to $T^* S^2\setminus S^2$ as follows 
\begin{align}
\mathcal{L} : \mathbb{C}^2 \setminus (\mathbb{C} \times \{0\})& \longrightarrow T^* \mathbb{C} \setminus \mathbb{C} \\ \nonumber
(u,v) &\mapsto (\dfrac{u}{\bar{v}}, 2v^2)
\end{align} 
where $\bar{v}$ is the complex conjugate of $v$. 

This regularization works only for a 2-dimensional space, i.e. $\mathbb{C}^2$. Note that, there is a higher dimension as Levi-Civita but we do not consider here \cite{key-11}. 

We consider a 2-dimensional space and discuss the Levi-Civita transformation. We extend the Levi-Civita regularization $\mathcal{L}$ to the cotangent bundle $T^* S^2$ as follows 
\begin{align}
\mathcal{L}: \mathbb{C}^2 \setminus \{0\} \longrightarrow T^* S^2 \setminus S^2
\end{align}
where $\mathbb{C}$ is assumed to be a chart of $S^2$ via stereographic projection as the north pole. The above extension gives us a covering map with degree 2. ( See \cite{key-2} for more details). 
\begin{lemma}
A closed hypersurface $\Sigma T^* S^2$ is fiberwise star-shaped if and only if $\mathcal{L}^{-1} \Sigma \subset \mathbb{C}^2$ is star-shaped. 
\end{lemma}
\begin{corollary}
There exists a diffeomorphism between a fiberwise star-shaped hypersurface in $T^* S^2$ and the projective space $\mathbb{R}P^3$ if $\mathcal{L} \Sigma \subset \mathbb{C}^2$ is star-shaped. 
\end{corollary}
Note that, a star-shaped hypersurface in $\mathbb{C} $ is diffeomorphic to the 3-dimensional sphere $S^3$ which is a twofold cover of $\mathbb{R}P^3$. 
\begin{example}
We apply Levi-Civita regularization to the Kepler problem. The Hamiltonian with respect to $u$ and $v$ is:
\begin{align}
H(u,v) = \dfrac{|u|^2}{2|v|^2} - \dfrac{1}{2|v|^2} -c,
\end{align}
where $c$ is the energy value. 

Based on the relation mentioned above, we can derive the following definition. 
\begin{align}
H'(u,v):= |v|^2 H(u,v) = \dfrac{1}{2}(|u|^2-c|v|^2 -1)
\end{align}
and for the energy zero, we have the level set 
\begin{align}
\Sigma:= H^{-1}(0) = H^{'-1}(0).
\end{align}
This is a 3-dimensional sphere for a negative energy $c$. 

The Hamiltonian flow of $H'$ on $\Sigma$ is just a parametrization of the Hamiltonian flow $H$ on $\Sigma$. The new Hamiltonian flow is periodic and physically, it is the flow of two uncoupled harmonic oscillators. 
\end{example}

\section{The special concave toric domain for the rotating Kepler problem}
In this section, we will introduce a Special Concave Toric Domain (SCTD) that is specifically designed for the RKP. This domain is both concave and toric, and we will examine its unique properties later on. Hutchings and his colleagues  have done extensive work on the concavity of toric domains, which has proven to be useful for calculating ECH capacities and answering symplectic embedding questions \cite{key-12}. Lastly, you will see the first significant aspect of my project, which aims to comprehend symplectic embedding questions for the R3BP with small mass ratios.

To compute this concave toric domain, we use the stereographic projection and transfer the cotangent bundle of $\mathbb{R}^2$ to the cotangent bundle of $S^2$. 

Since the Ligon-Schaaf symplectomorphism interchanges the Hamiltonian of the Kepler problem with the Delaunay Hamiltonian. Therefore, we get the solution of the Kepler problem as geodesics on the cotangent bundle $T^*S^2$.

Angular momentum is responsible for generating rotation. When it comes to the RKP, its Hamiltonian is obtained by adding angular momentum to the Kepler problem's Hamiltonian. The Ligon-Schaaf symplectomorphism interchanges the angular momentum on the plane with an angular momentum component on the sphere. As a result, the Ligon-Schaaf symplectomorphism pulls back the Hamiltonian of the RKP to a Hamiltonian defined on the cotangent bundle $S^2$ minus its zero section. The Levi-Civita map is a 2:1 map between $\mathbb{C}^2$ minus the origin and the cotangent bundle of $S^2$ minus the zero section.

Let's start by considering the phase space $T$ of the geodesic solutions of the RKP. Using Levi-Civita regularization, we can map this phase space to the space $\mathbb{C}^2$. This mapping results in a double cover, which allows us to define a special concave toric domain that is well-define for the RKP. 

\subsection{Construction} 
Here, we will use the above notation and definitions to compute the concave toric domain of the RKP in several steps.

Given the unit sphere $S^2$ and denote the north pole of it in $\mathbb{R}^3$ with $N=(0,0,1)$. Take a point $x=(x_1,x_2,x_3)$ on $S^2$ and a covector on the tangent space of $S^2$ at $x$ with $y=(y_1,y_2,y_3)$ such that $x \neq N$, $x\cdot x =1$ and $ x \cdot y =0$. 

Now we use the stereographic transformation and map the cotangent bundle of the space $\mathbb{R}^2$ to the cotangent bundle of the sphere $S^2$. In other words, we have 
\begin{align}
T^* \mathbb{R}^2 &\longrightarrow T^* S^2 \\ \nonumber
(q,p) & \mapsto (x,y). 
\end{align}
such that the following equalities hold
\begin{align} \label{c1}
x_k = & \dfrac{2q_k}{(q^2 +1)} \, , \qquad x_3 =\dfrac{(q^2-1)}{(q^2 +1)} \\ \nonumber
y_k =& \dfrac{(q+1)p_k}{2} - (q \cdot p) q_k \, , \qquad y_3 = q \cdot p \nonumber
\end{align}
where $k=1,2$. 

These are canonical transformations in the sense that the symplectic form $\Sigma_{k=1}^2 dq_k \wedge dp_k$ and the restriction of $\Sigma_{k=1}^3 dx_k\wedge dy_k $ to $ T^*S^2$ match. 
Given the Delaunay Hamiltonian 
\begin{align}
\tilde{H} (x,y) = -\dfrac{1}{|| 2y^2||},
\end{align}
where $||.||$ is the norm respect to the round geometry of $S^2$. Note that the Hamiltonian flow of the Delaunay Hamiltonian is a reparametrized geodesic flow on $S^2$. 

Applying the stereographic projection \ref{c1} to the Delaunay Hamiltonian becomes 
\begin{align}
\tilde{H}(q,p)= -\dfrac{2}{(|q|+1)^2 |p|^2}.
\end{align}
The property 
\begin{align}
\Phi_{LS}^* H =\tilde{H},
\end{align}
of the Ligon-Schaaf symplectomorphism guarantees the Ligon-Schaaf symplectomorphism maps the Hamiltonian vector field of the Kepler problem to the Hamiltonian vector field of the Delaunay Hamiltonian. 

The Ligon-Schaaf interchanges angular momentum in $\mathbb{R}^2$ with the first component of the angular momentum on $S^2$. Therefore, applying the Ligon-Schaaf symplectomorphism and the stereographic projection on the Hamiltonian of the RKP becomes 
\begin{align}
K(q,p) = \tilde{H}(q,p) + L(q,p) = -\dfrac{2}{(|q|+1)^2|p|^2}+q_1p_2 - q_2p_1.
\end{align}
If we let $q$ and $p$ as complex numbers, i.e. $q=q_1+iq_2$ and $p=p_1+ip_2$. We can write 
\begin{align}\label{2}
K(q,p) = \tilde{H}(q,p) + L(q,p) = -\dfrac{2}{(|q|+1)^2|p|^2} + Im(\bar{q} \cdot p).
\end{align}
We know that the Levi-Civita transformation is a 2:1 map which up to a constant factor is symplectic when we think of $\mathbb{C}^2$ as $T\mathbb{C}$. It pulls back the geodesics flow on $S^2$ to the flow of two uncoupled oscillators. 

We apply the Levi-Civita regularization. To this purpose, we apply $\dfrac{u}{\bar{v}}$ and $2v^2$ in the relation \ref{2} instead of $q$ and $p$ respectively. Then we get the following identity 
\begin{align} \label{c5}
\tilde{H}(u,v) +L(u,v) & = - \dfrac{2}{(|\frac{u}{\bar{v}}| +1)^2 (|2v^2|)^2} +Im(\dfrac{\bar{u}}{\bar{v}}\cdot 2v^2) \\ \nonumber
&= -\dfrac{2}{2(|u|^2+|v|^2)^2}+ 2 Im(\bar{u}v) \\ \nonumber
& = -\dfrac{1}{2(|u|^2 +|v|^2)^2}+2(u_1v_2 - u_2v_1). 
\end{align}
We introduce the function 
\begin{align}
\mu : T^* \mathbb{C} & \longrightarrow [0, \infty) \times \mathbb{R} \subset \mathbb{R}^2 \\
(u,v) & \mapsto 
\begin{cases}
\frac{1}{2} (|u|^2+|v|^2)\\ 
u_1v_2 - u_2v_1.
\end{cases}
\end{align}
This is the momentum map of the torus action on $T^* \mathbb{C}$. 

Note that in view of the elementary inequality 
\begin{align}
|ab| \leq \dfrac{1}{2} (a^2+b^2),
\end{align}
follows that $|\mu_2| \leq \mu_1$.Therefore, componentwise we define 
\begin{align}
\mu_1:= \dfrac{|u|^2+|v|^2}{2}, \:\:\:\ \mu_2 := u_1v_2 +u_2v_1. 
\end{align}
If we define the Hamiltonian of the RKP with $K$ and use the above notation and definitions, then we have the following proposition. 
\begin{proposition}
Given the Ligon-Schaaf symplectomorphism and the Levi-Civita regularization, the pullback of $K$ becomes
\begin{align}\label{c6}
\mathcal{L}^* \Phi_{LS}^* (K) = -\dfrac{1}{8\mu_1^2} + 2 \mu_2.
\end{align}
\end{proposition}
\begin{proof}
This follows from the discussion above. 
\end{proof}
We show that the symplectic manifold $\mathbb{C}\oplus \mathbb{C}$ and the cotangent bundle $T^*\mathbb{C}$ are symplectomorphic.

\begin{proposition}\label{Pro2}
There exists a linear symplectomorphism between the symplectic manifold $\mathbb{C} \oplus \mathbb{C}$ and the cotangent bundle $T^\ast \mathbb{C}$. In other words, we have the linear symplectomorphism 
\begin{align}
\mathit{S}: (\mathbb{C} \oplus \mathbb{C},\omega_0) \longrightarrow (T^\ast \mathbb{C},\omega_1).
\end{align}
\end{proposition}
\begin{proof}
Consider the symplectic form on $T^\ast \mathbb{C}$ as 
\begin{align}
\omega_1 = du_1 \wedge dv_1 + du_2 \wedge dv_2. 
\end{align}
Let $(z_1,z_2) \in \mathbb{C}^2 $ such that $z_1=x_1+iy_1 $ and $z_2=x_2+iy_2$. We define the following linear map 
\begin{align}
\mathit{S}: \mathbb{C}^2 \longrightarrow T^\ast \mathbb{C}
\end{align}
as 
\begin{align} \label{16}
u_1 & \longrightarrow \frac{1}{\sqrt{2}}(y_1-y_2) \\ \nonumber 
u_2 & \longrightarrow \frac{1}{\sqrt{2}}(x_1+x_2) \\ \nonumber
v_1 & \longrightarrow \frac{1}{\sqrt{2}}(x_2-x_1) \\ \nonumber
v_2 & \longrightarrow \frac{1}{\sqrt{2}}(y_1+y_2) \\ \nonumber
\end{align}
To demonstrate that $S$ interchanges the symplectic forms $\omega_0$ and $\omega_1$, we can compute using \ref{16}. 
 Thus we have 
\begin{align}
\mathit{S}^\ast (\omega_1) & = \mathit{S}^\ast (du_1 \wedge dv_1 + du_2 \wedge dv_2) \\ \nonumber
& = (\frac{1}{\sqrt{2}} (dy_1-dy_2) \wedge \frac{1}{\sqrt{2}} (dx_2-dx_1) ) +(\frac{1}{\sqrt{2}} (dx_1-dx_2) \wedge \frac{1}{\sqrt{2}} (dy_1+ dy_2)) \\ \nonumber
&= dx_1 \wedge dy_1 + dx_2 \wedge dy_2 \\ \nonumber
& =\omega_0.
\end{align}
\end{proof}
We extend the function \ref{c6} to $T^* \mathbb{C} \setminus \{0\}$ and define 
\begin{align} \label{c7}
\tilde{K}:& T^* \mathbb{C} \setminus \{0\} \longrightarrow \mathbb{R} \\
\tilde{K}:&= \dfrac{-1}{8 \mu_1^2} + 2 \mu_2. \nonumber
\end{align}
Using the above function gives us a concave toric domain for the RKP on a coordinate system which is rotated in view of the proposition \ref{Pro2}.

We make the following rotations. Denote the first quarter in $\mathbb{R}^2$ by $Q:= [0,\infty) \times [0,\infty)$ and define 
\begin{align}
Q_{\frac{1}{2}} := \{(x,y) \in \mathbb{R}^2 \: : \: x \geq 0\; , |y| < x \},
\end{align}
where $Q_{\frac{1}{2}}$  is the first quadrant of $\mathbb{R}^2$ that rotated by 45 degree clockwise. 

Assume $\Omega \subset Q$ is close in the first quarter in $\mathbb{R}^2$. A toric domain is defined by \begin{align}
X_\Omega := \nu^{-1} (\Omega),
\end{align}
where
\begin{align}
\nu=(\nu_1 , \nu_2) : \mathbb{C}^2 &\longrightarrow Q \subset \mathbb{R}^2 \\
(z_1 ,z_2)& \mapsto (\pi |z_1|^2 , \pi|z_2|^2).  \nonumber
\end{align}
Note that $\nu$ is a momentum map for the torus action 
\begin{align}
(\nu_1 , \nu_2)(z_1,z_2) = (e^{i \theta_1} z_1 , e^{i\theta_2}z_2)
\end{align}
on $\mathbb{C}^2$. 
We define the symplectic 4-manifold with boundary as 
\begin{align}
X_\Omega := \{ z=(z_1,z_2) \in \mathbb{C}^2 \;:\; \pi (|z_1|^2 , |z_2|^2) \in \Omega \}.
\end{align}
$\mathbf{Recall:}$ The concave toric domain according Hutchings \cite{key-12} is defined as follow.
\begin{definition}
 We say that a toric domain 
$X_\Omega$
is a concave toric domain if $\Omega$ is a closed region bounded by the horizontal segment from
$(0,0)$
to
$(a,0)$,
the vertical segment from $(0,0)$ to $(0,b)$ and graph of a convex function 
$f:[0,a] \longrightarrow [0,b]$ with
$f(0)=b$
and 
$f(a)=0$,
where 
$a>0 $ 
and 
$b>0$.
\end{definition}
Now we define a Special Concave Toric Domain (SCTD) as follows.
\begin{definition}
A concave toric domain $X_\Omega \subset \mathbb{C}^2$ is called special if the function $f$ satisfies the additional property $f'(t) \geq -1$ for $t \in [0, a]$.
\end{definition} 
We compare the Hutchings Concave Toric Domain (CTD) and the SCTD for the RKP. 

Define $\bar{S}$ by 
\begin{align}
\bar{S}: Q \longrightarrow Q_{\frac{1}{2}}
\end{align}
which is a clockwise 45 degree rotation composed with a $\dfrac{1}{\sqrt{2} \pi}$ dilation.

Using the above notation, we can obtain the following relations between momentum maps $\nu$ and $\mu$  for the torus actions $\mathbb{C}^2$  and $T^* \mathbb{C}^2$.  
 
\begin{align}
\bar{S}(\dfrac{1}{2 \pi} (\nu_1+\nu_2)) & = \mu_1 \\ \nonumber
\bar{S}(\dfrac{1}{2 \pi} (\nu_1-\nu_2)) & = \mu_2.
\end{align}
Using these equalities gives us the following commutative diagram 

\begin{align}
\begin{array}{ccc} 
\mathbb{C} \oplus \mathbb{C} & \xrightarrow{S} & T^*\mathbb{C} \\
{\nu}\downarrow && \downarrow{\mu} \\
Q & \xrightarrow{\bar{S}} & Q_{\frac{1}{2}}
\end{array}
\end{align}

We alternatively define a concave toric domain for $\Omega' $ 
\begin{align}
\Omega ' := \bar{S} (\Omega) \subset Q_{\frac{1}{2}},
\end{align}
by 
\begin{align}
X_{\Omega'} = \mu^{-1} (\Omega') = S(X_\Omega)
\end{align}
in $T^*\mathbb{C}$. 

Assuming the concave toric domain as a subset of $T^* \mathbb{C}$ instead of $\mathbb{C}^2$, we consider $\Omega$ as a closed subset of $\Q_{\frac{1}{2}}$, without the prime.

Using the new convention, the SCTD can be defined as follows. 
\begin{remark}\label{c9}
Let's use the above notations of $\mathbb{C}^2$ and $T^\ast \mathbb{C}$ to make things clearer. A   toric domain $X_\Omega$ is a special concave toric domain if and only if there exists a convex function 
\begin{align}
g : [a,b] \longrightarrow \mathbb{R}, \qquad 0 < a<b<\infty ,
\end{align}
with properties $g(a)=a$, $g(b) = -b$
such that $\Omega \subset Q_{\frac{1}{2}}$ is bounded by the segment 
$\{ (t,t ): t \in [0,a]\} , \{(t,-t) : t \in [0,b]\}$
and the graph of the convex function $g$. 
\end{remark}
\begin{remark} \label{c10}
We are working with $\Omega \subset Q_{\frac{1}{2}}$. If $\Omega$ satisfies the conditions of remark \ref{c9} we refer to 
$ X_\Omega := \mu^{-1} (\Omega)$ as a special concave toric domain. 

Assume $ c  \leq -\frac{3}{2}$, we define a closed subset of $Q_{\frac{1}{2}}$ by 
\begin{align}
\mathcal{K}_c := \mu (\tilde{K}^{-1}( -\infty , c)) \subset Q_{\frac{1}{2}}.
\end{align}
Note that if $c < -\dfrac{3}{2}$ then $\mathcal{K}_c $ has two connected components, one bounded and one unbounded, i.e. we write 
\begin{align}
\mathcal{K}_c = \mathcal{K}_c^b \cup \mathcal{K}_c^u,
\end{align}
for $\mathcal{K}_c^b$ the bounded connected component and $\mathcal{K}_c^u$ the unbounded connected component. 

For $c =-\dfrac{3}{2}$ the two sets become connected at singularity which is the point $(\frac{1}{2} , -\frac{1}{2})$. 
\end{remark}
\begin{theorem}
For $c \leq - \frac{3}{2}$, we have
\begin{align}
\tilde{K}^{-1}( -\infty, c) = X_{\mathcal{K}_c^b} \cup X_{\mathcal{K}_c^u} \subset T^\ast \mathbb{C}
\end{align}
and 
$X_{\mathcal{K}_c^b}$ is the SCTD. 
\end{theorem}
\begin{proof}
After applying all the necessary transformations, which are explained in the above, the statement can be easily derived from reference \ref{c7}, given the function
$(0 , \infty) \longrightarrow \mathbb{R}$
\begin{align*}
x \mapsto \dfrac{1}{16 x^2} 
\end{align*}
is convex. 
\end{proof}
We will see the graphs of the SCTD for the energies $c\leq -\dfrac{3}{2}$, $c =-\dfrac{3}{2}$ and $c> -\dfrac{3}{2}$ in the following figures. 

 \begin{figure}[htbp]
\includegraphics[scale=0.15]{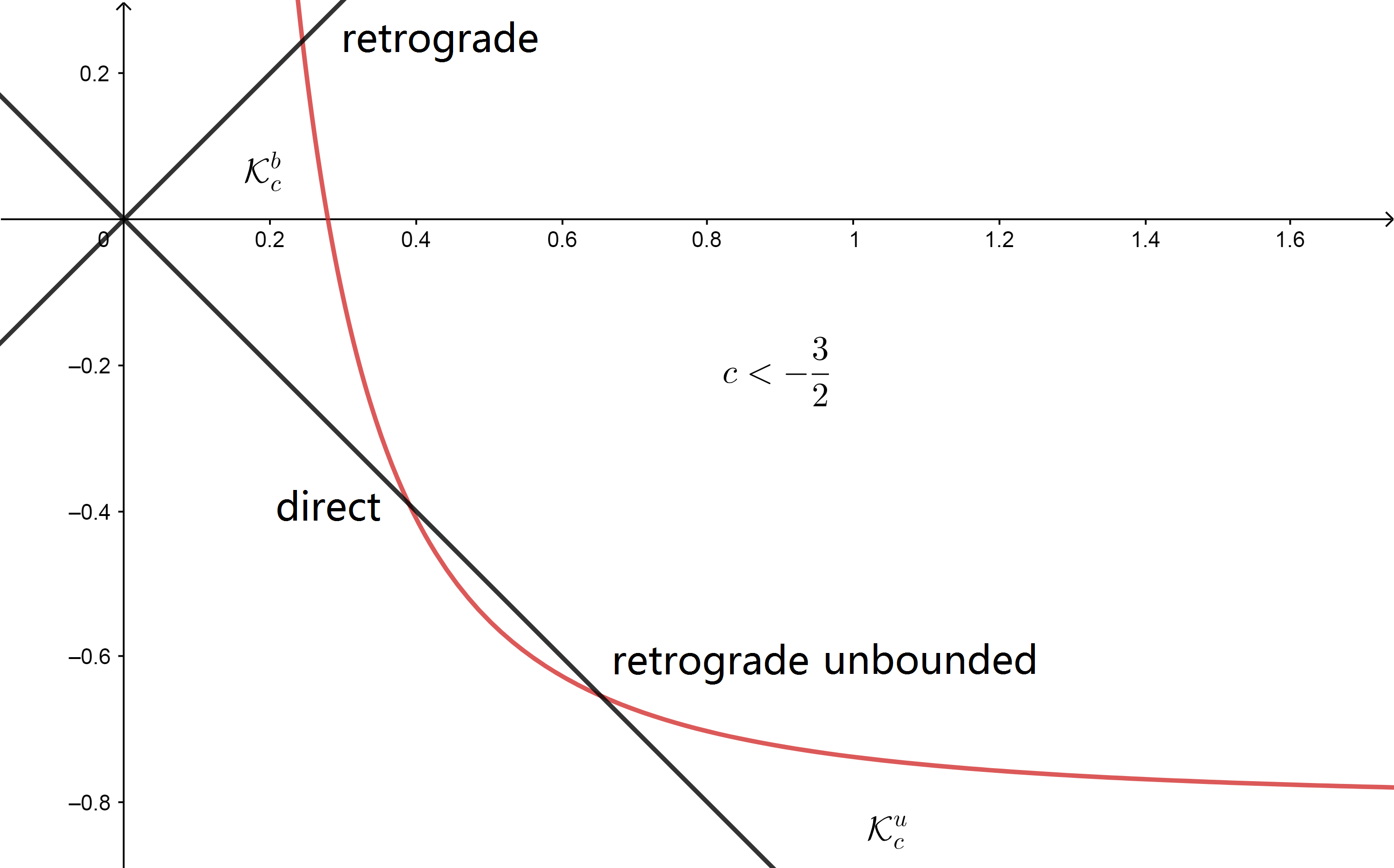}
  \captionof{figure}{The direct  and the  retrograde orbits for an energy $c <-\dfrac{3}{2}$}
   \label{picSCTD4}
  \end{figure}
 
\begin{figure}[htbp]
\includegraphics[scale=0.15]{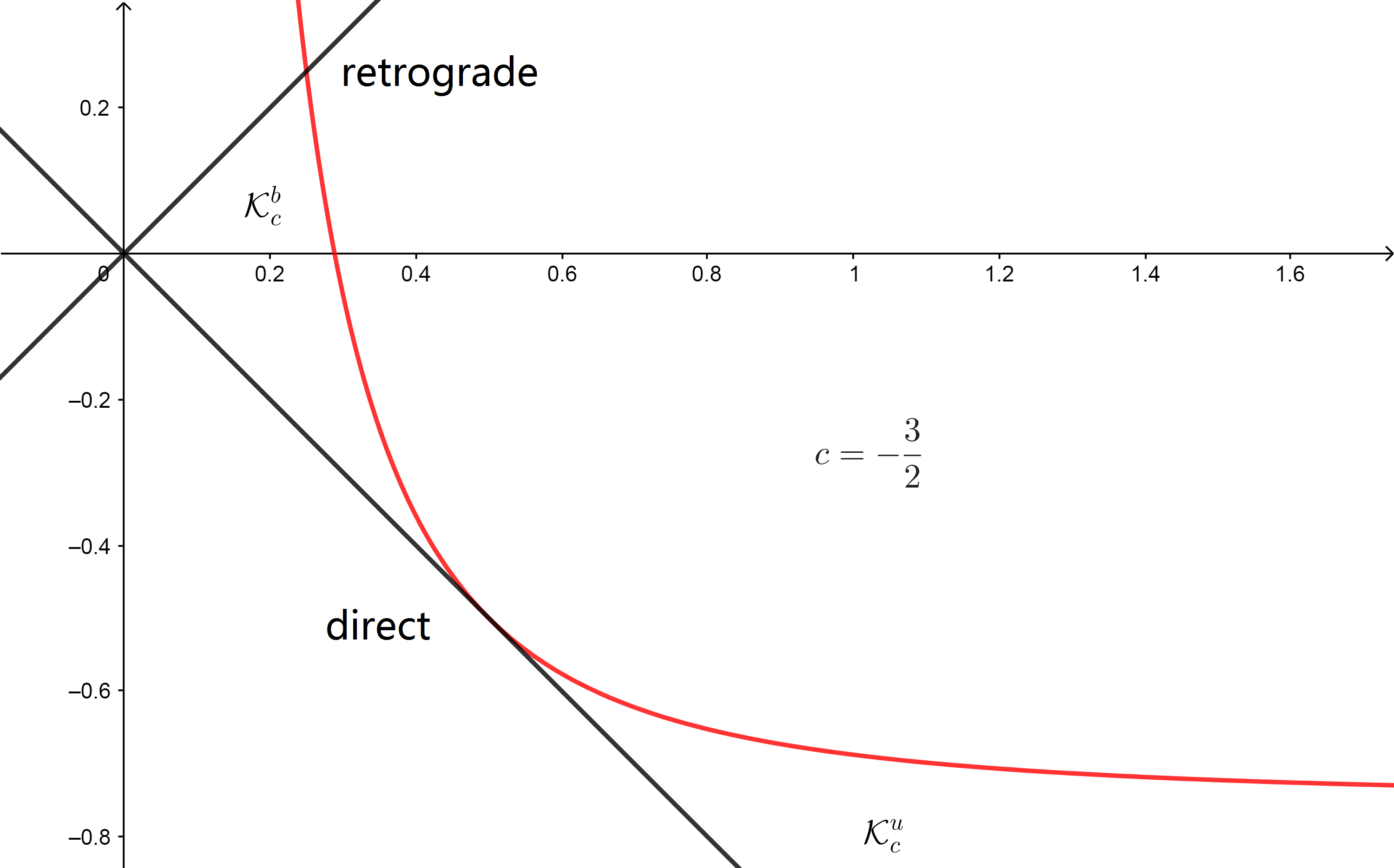}
  \captionof{figure}{The  direct orbit for the energy $c =-\dfrac{3}{2}$}
  \label{picSCTD5}
\end{figure}

\begin{figure}[htbp]
\includegraphics[scale=0.15]{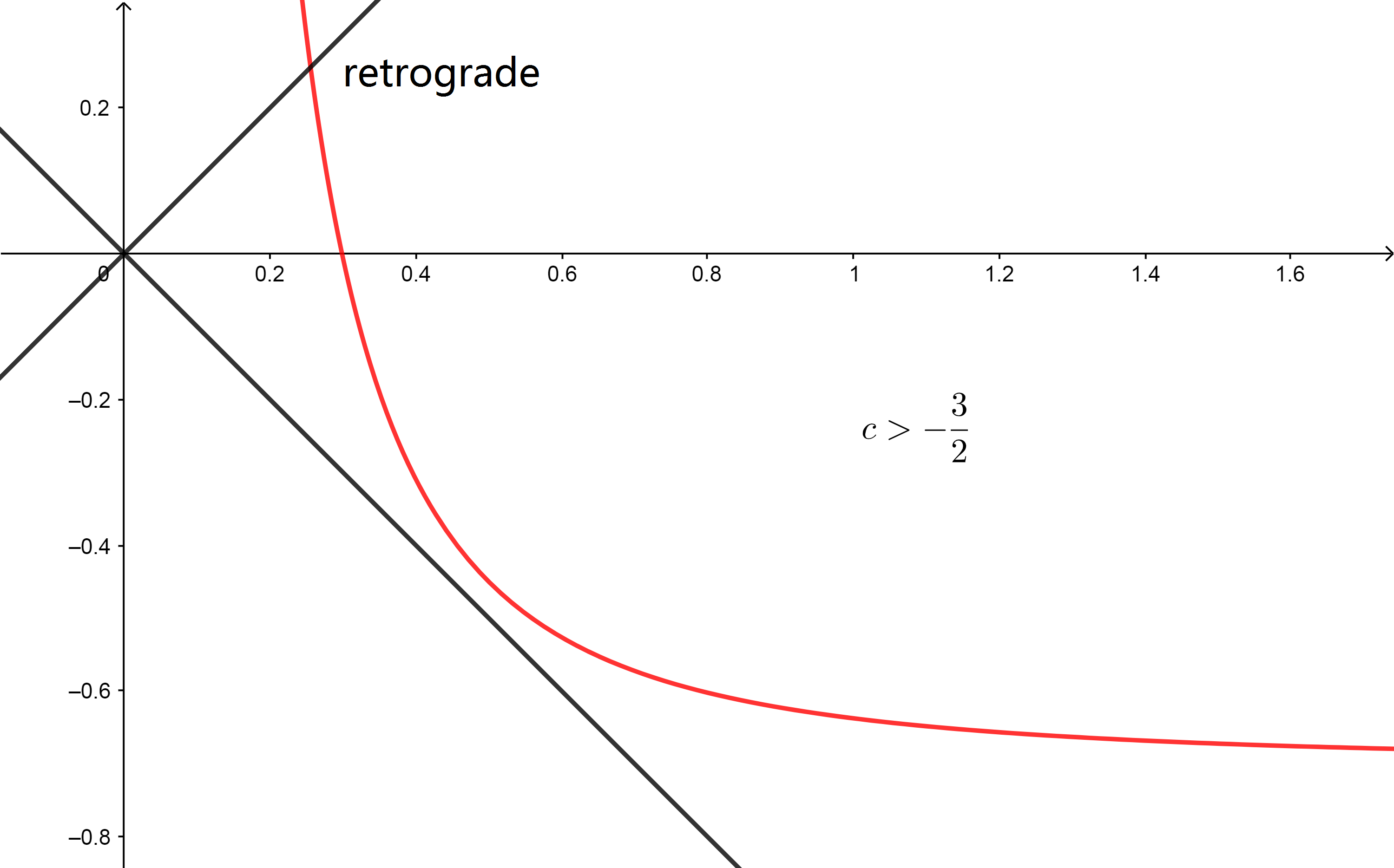}  
 \captionof{figure}{There is no    direct orbit  for  energy $c >-\dfrac{3}{2}$}
  \label{picSCTD6}
  \end{figure}
  \newpage
\section{Construction of a new tree:}
Now, we are introducing a new tree that can be used to calculate the slopes of the SCTD. These slopes will be useful in future SCTD work. The new tree is an extension of the Stern-Brocot tree, which we will explain shortly. To start, let us briefly recap the Stern-Brocot tree.

\subsection{ The Stern-Brocot tree} The Stern-Brocot tree was introduced by Moritz-Stern 1858 and Achille Brocot 1861. The Stern-Brocot tree is a complete infinity binary tree whose nodes are labeled by a unique rational number. 

There is more information about this tree and the Calkin-Wilf tree and their relations in \cite{key-2}.

We use induction and a mediant method to construct the Stern-Brocot tree. Another way to obtain the tree is via the Calkin-Wilf tree \cite{key-2}.

\begin{definition}
A mediant is a fraction such that its numerator is the sum of the numerators of two other fractions and its denominator is the sum of the denominators of two other fractions.  
\end{definition} 
The Stern-Brocot tree begins at level -1, with the pseudo-fractions $\dfrac{0}{1}$ and $\dfrac{1}{0}$. To generate a new level, we use the previous level and the mediants to create new fractions. These new fractions are then arranged in increasing order on a line, and this process is repeated to generate subsequent levels of the Stern-Brocot tree. In other words, we can express this as an induction process.
\begin{enumerate} 
\item[Stage -1:]
We start with the auxiliary  labels $\frac{0}{1}$ and $\frac{1}{0}$ lowest to highest terms. Stage -1, we do not really consider as part of the tree but this level is used in the inductive constructive of the tree. \\
\item[Stage 0:] The root is $\dfrac{1}{1}$ which can be interpreted as mendiant of stage -1. \\
\item[Stage 1:]
We add the mediant of the boundaries. \\
\item[Stage n+1:]
We add the mediants of all consecutive fractions in the tree including the boundaries from the lowest to highest.
\end{enumerate}
Therefore we have the following tree. 

\begin{center} 
\begin{tikzpicture}[<-,->=stealth', auto, semithick,
level/.style={sibling distance=60mm/#1}
]
\node (z){$\dfrac{1}{1}$}
child {node (a) {$\dfrac{1}{2}$}
child {node (b) {$\dfrac{1}{3}$}
child {node {$\dfrac{1}{4}$}
child {node {$\vdots$}} 
} 
child {node {$\dfrac{2}{5}$}
 child {node {$\vdots$}}}
}
child {node (g) {$\dfrac{2}{3}$}
child {node {$\dfrac{3}{5}$}
child {node {$\vdots$}}}
child {node {$\dfrac{3}{4}$}
child {node {$\vdots$}}}
}
}
child {node (j) {$\dfrac{2}{1}$}
child {node (k) {$\dfrac{3}{2}$}
child {node {$\dfrac{4}{3}$}
child {node {$\vdots$}}}
child {node {$\dfrac{5}{3}$}
child {node {$\vdots$}}}
}
child {node (l) {$\dfrac{3}{1}$}
child {node {$\dfrac{5}{2}$}
child {node {$\vdots$}}}
child {node (c){$\dfrac{4}{1}$}
child {node {$\vdots$}}
}
}
};
child [grow=down] {
edge from parent[draw=none]
};
(z) -- ++(0,1cm);
\end{tikzpicture}
\end{center}

For convenience, we define a labeling for the nodes of the tree. We start from a root-like $\dfrac{a}{b}$ and try to write the next level according to the previous one such that the lowest node is the first and the highest one is the last node. In the new level, each root like $\dfrac{a}{b}$ has two children such that the left child is less than $\dfrac{a}{b}$ and the right child is bigger than $\dfrac{a}{b}$. We called the left child even and denote it with zero and we called the right child odd and denote it with 1. We use these notations to find the nodes on the Stern-Brocot tree and also determine the place of new tree nodes. 
\subsection{The New Tree:}
Using the above notations, we explain the new tree which is useful to find the slopes and critical energy values of tori and asteroids in the SCTD. Consider the labeling of the Stern-Brocot tree and let the noes of the Stern-Brocot tree by the fractional number $\dfrac{k}{l}$. We write the node $\dfrac{k}{l}$ as a matrix $\left[
\begin{array}{ccc}
k \\
l 
\end{array}
\right]
$. Since we want to have the new tree on the rotated coordinate by 45 degree, we multiply the matrix of the node $\dfrac{k}{l}$ by $ 
\left[
\begin{array}{ccc}
1 & 1 \\
-1 & 1 
\end{array}
\right]$   which corresponds to a rotation by 45 degrees and a dilation by $\sqrt{2}$ in the coordinate system. Note that multiplicity and dilation do not influence the slope. Namely, 
\begin{align} 
\left[
\begin{array}{ccc}
1 & 1 \\
-1 & 1 
\end{array}
\right] 
\left[
\begin{array}{ccc}
k \\
l 
\end{array}
\right] = \dfrac{k+l}{-k+l}.
\end{align}
Now we replace the node $\dfrac{k}{l}$ in the Stern-Bropcot tree by the node $\dfrac{k+l}{-k+l}$. They follow the above method for all nodes of the Stern-Brocot tree to get the new tree. The new tree is 

\begin{center} 
\begin{tikzpicture}[<-,->=stealth', auto, semithick,
level/.style={sibling distance=60mm/#1}
]
\node (z){$\infty $}
child {node (a) {$\dfrac{3}{1}$}
child {node (b) {$\dfrac{2}{1}$}
child {node {$\dfrac{5}{3}$}
child {node {$\vdots$}} 
} 
child {node {$\dfrac{7}{3}$}
 child {node {$\vdots$}}}
}
child {node (g) {$\dfrac{5}{1}$}
child {node {$\dfrac{4}{1}$}
child {node {$\vdots$}}}
child {node {$\dfrac{7}{1}$}
child {node {$\vdots$}}}
}
}
child {node (j) {$-\dfrac{3}{1}$}
child {node (k) {$-\dfrac{5}{1}$}
child {node {$-\dfrac{7}{1}$}
child {node {$\vdots$}}}
child {node {$-\dfrac{4}{1}$}
child {node {$\vdots$}}}
}
child {node (l) {$-\dfrac{2}{1}$}
child {node {$-\dfrac{7}{3}$}
child {node {$\vdots$}}}
child {node (c){$-\dfrac{5}{3}$}
child {node {$\vdots$}}
}
}
};
child [grow=down] {
edge from parent[draw=none]
};
(z) -- ++(0,1cm).
\end{tikzpicture}
\end{center}

Note that the nodes of the new tree are the slopes of the tori $T_{k,l}$ in the SCTD. Therefore, the slopes in the SCTD are determined uniquely by a rational number from the new tree.

\end{document}